\RequirePackage{fix-cm}
\documentclass[smallextended]{svjour3}       
\smartqed  
\usepackage{amssymb,amsmath,multirow}
\usepackage{graphicx}

\newtheorem{thm}{Theorem}
\newtheorem{exam}{Example}

\begin{document}
\title{Approximation of the weighted maximin dispersion problem over $\ell_p-$ball: SDP relaxation
is misleading
\thanks{This research was supported  by National Natural Science Foundation of China
under grants 11471325 and 11571029, and by fundamental research funds
 for the Central Universities under
grant YWF-16-BJ-Y-11.}
}

\titlerunning{Weighted maximin dispersion problem over $\ell_p$-ball}        

\author{Zuping Wu\and Yong Xia\footnote{Corresponding author}  \and Shu Wang }


\institute{   Z. Wu, Y. Xia, S. Wang        \at
State Key Laboratory of Software Development
              Environment, LMIB of the Ministry of Education,
              School of
Mathematics and System Sciences, Beihang University, Beijing,
100191, P. R. China
              \email{wuzuping2008@sina.com(Z. Wu); dearyxia@gmail.com(Y. Xia);
              wangshu.0130@163.com(S. Wang)}
 }

\date{Received: date / Accepted: date}

\maketitle

\begin{abstract}
Consider the problem of finding a point in a unit $n$-dimensional $\ell_p$-ball ($p\ge 2$) such that the minimum of the weighted Euclidean distance from given $m$ points is maximized. We show in this paper that the recent SDP-relaxation-based approximation algorithm [SIAM J. Optim. 23(4), 2264-2294, 2013] will not only provide the first theoretical  approximation bound of $\frac{1-O\left(\sqrt{ \ln(m)/n}\right)}{2}$, but also perform much better in practice, if  the SDP relaxation is removed and  the optimal solution of the SDP relaxation is replaced by a simple scalar matrix.
\keywords{ Maximin dispersion\and Convex relaxation \and Semidefinite programming\and Approximation algorithm}
\subclass{90C20\and 90C26 \and 90C47 \and 68W25}
\end{abstract}

\section{Introduction}
As is well-known, following the pioneer work on providing a $0.878$-approximate solution for max-cut problem \cite{GW95}, the semidefinite programming (SDP) relaxation technique has been playing a great role in approximately solving  combinatorial optimization problems and nonconvex quadratic programs; see for example, \cite{GJ,Luo,Nem,Ne,Ye}.

This paper is to present a surprise case where the SDP relaxation misleads the approximation in both theory and computation.

Consider the  $\ell_p$-ball ($p\geq 2$) constrained weighted maximin dispersion problem:
\begin{eqnarray*}
{\rm (P)}~~\max_{\|x\|_p\le 1}\left\{ f(x):=\min_{i=1,\ldots,m} \omega_i\|x-x^i\|_2^2\right\},
\end{eqnarray*}
where
$x^1,\ldots,x^m\in \Bbb R^n$ are given $m$ points, $\omega_i>0$ for $i=1,\ldots,m$, and $\|x\|_p:=\left(\sum_{i=1}^n|x_i|^p\right)^{\frac{1}{p}}$ is the $\ell_p$-norm of $x$. Applications of (P) can be found in facility location, spatial management, and pattern recognition; see \cite{DW,JMY,Sc,W}
and references therein.

Based on the SDP relaxation technique, Haines et al. \cite{HA13} proposed the first approximation algorithm for solving (P). \footnote{Their algorithm is actually proposed for the weighted maximin dispersion problem with a more general constraint.} However, their approximation bound is not so clean that it depends on the optimal solution of the SDP relaxation. Fortunately, when $p=+\infty$, the approximation bound reduces to
\begin{equation}
\frac{1-O\left(\sqrt{ \frac{\ln(m)}{n}}\right)}{2}. \label{bd}
\end{equation}
Very recently, the above approximation bound (\ref{bd}) is established for the special case $p=2$ based on a different algorithm \cite{WX16}. But further extension to (P) with $p>2$  remains open  \cite{WX16}.

In this paper,  we show that, by removing the SDP relaxation from Haines et al.'s approximation algorithm \cite{HA13} and simply replacing the optimal solution of the SDP relaxation with a scalar matrix,
the approximation bound (\ref{bd}) becomes to be satisfied for (P). It is the SDP relaxation that makes the whole approximation algorithm not only loses the theoretical bound (\ref{bd}) but also performs poorly in practice.

The remainder of this paper is organized as follows. In Section 2, we present the existing approximation algorithm based on SDP relaxation. In Section 3, we propose a new simple approximation algorithm without any convex relaxation and establish the approximation bound. Numerical comparison is reported in Section 4. We make conclusions in Section 5.

Throughout this paper, we denote by $\Bbb R^n$ and $S^n$ the $n$-dimensional real vector space and the space of $n\times n$ real symmetric matrices,
respectively. Let $I_n$ be the identity matrix of order $n$. $A\succ(\succeq)0$ denotes that $A$  is positive (semi)definite. The inner product of two matrices $A$ and $B$ is denoted by $A\bullet B={\rm Tr}(AB^T)=\sum_{i=1}^n\sum_{j=1}^na_{ij}b_{ij}$. ${\rm Pr(\cdot)}$ stands for the probability.

\section{Approximation algorithm based on SDP relaxation}
In this section, we present Haines et al.'s randomized approximation algorithm \cite{HA13} based on SDP relaxation.

It is not difficult to verify that lifting $xx^T$ to $X\in S^{n}$ yields the SDP relaxation for (P):
\begin{eqnarray*}
{\rm (SDP)}~~&\max_{X, x}&~\min\left\{
\omega_i  \left(\begin{array}{cc}I_n & -x^i\\-(x^i)^T & \|x^i\|_2^2\end{array}\right)\bullet \left(\begin{array}{cc}X & x\\x^T & 1\end{array}\right)\right\} \\
&{\rm s.~t.}&
X_{11}^{\frac{p}{2}}+X_{22}^{\frac{p}{2}}+\ldots+X_{nn}^{\frac{p}{2}}\leq 1, \\
&& \left(\begin{array}{cc}X & x\\x^T & 1\end{array}\right)\succeq 0.
\end{eqnarray*}
Since it is assumed that $p\ge 2$, the above (SDP) is a convx program and hence can be solved efficiently \cite{NN}.

Now, we present Haines et al.'s SDP-based approximation algorithm \cite{HA13}
for solving (P).

~

\begin{center}
\fbox{\shortstack[l]{
{\bf Algorithm 1: SDP-based approximation algorithm \cite{HA13}.}\\
1.~ Input $\rho\in(0,1)$ and $x^i$ for $i=1,\ldots,m$. Let $\alpha=\sqrt{2\ln(m/\rho)}$. \\
2.~
 Solve ${\rm (SDP)}$  and return the optimal solution $X^*\in S^{n}$. Set\\~~~~~ $b^i=(\sqrt{X^*_{11}}x^i_1,\ldots,\sqrt{X^*_{nn}}x^i_n)^T$  for $i\in\{1,\ldots,m\}\setminus\{k\mid\|x^k\|=0\}$.\\
3.~    Repeatedly generate $\xi=(\xi_1,\ldots,\xi_n)^T$ with independent $\xi_i$ taking \\~~~~~the value $\pm1$ with equal probability until  $(b^i)^T\xi<\alpha\|b^i\|$ for \\~~~~~ $i\in\{1,\ldots,m\}\setminus\{k\mid\|x^k\|=0\}$.\\
4.~ Output
$\widetilde{x}=\left(\sqrt{{X^*_{11}}}\xi_1,\sqrt{{X^*_{22}}}\xi_2,\ldots,\sqrt{{X^*_{nn}}}\xi_n\right)^T$.
}}
\end{center}

~

\begin{remark}
The original version of the above algorithm \cite{HA13} did not consider the possible case $\|x^k\|=0$ for some $k$. It has been fixed in \cite{WX16}.
\end{remark}

\begin{remark}
The existence of $\xi$  in Step 3 of Algorithm $1$ is guaranteed by the inequality \cite{HA13}:
\[
{\rm Pr}\left((b^i)^T\xi<\alpha\|b^i\|,~i\in I \right)\ge 1-\rho>0,
\]
which is a trivial corollary of the following well-known result.
\begin{thm} {\rm \cite[Lemma A.3]{Ben2002}}\label{thm2}
Let $\xi\in \{-1,1\}^n$ be a random vector, componentwise independent, with
\[
{\rm Pr}(\xi_j=1)={\rm Pr}(\xi_j=-1)=\frac{1}{2},~\forall j=1,\ldots,n.
\]
Let $b\in \Bbb R^n$ and $\|b\|> 0$. Then, for any $\alpha>0$,
\[
{\rm Pr}(b^T\xi\geq \alpha\|b\|)\leq e^{-\alpha^2/2}.
\]
\end{thm}
\end{remark}

For the approximation bound of Algorithm 1, the following main result holds.
\begin{thm} {\rm\cite[Theorem 3]{HA13}}\label{thm0}
For the solution $\widetilde{x}$ returned by the Algorithm 1,  we have
\[
v({\rm P})\ge f(\widetilde{x})\geq \frac{1-\sqrt{2\gamma_1^*\ln(m/\rho)} }{2}v({\rm P}).
\]
where $\gamma_1^*=\frac{\max_{j=1,\ldots,n}X^*_{jj}}{\sum_{j=1}^nX^*_{jj}}$. Moreover, when $p=+\infty$, $\gamma_1^*=\frac{1}{n}$.
\end{thm}

Finally, we remark that the equality $\gamma_1^*=\frac{1}{n}$ is no longer true when $p<+\infty$. The following counterexample is taking from \cite{WX16}.
\begin{exam}{\rm\cite[Example 4.1]{WX16}}
Let $n=2,~m=3$, $p=2$, $x^1=\left(\begin{array}{c}1\\2\end{array}\right)$, $x^2=\left(\begin{array}{c}2\\3\end{array}\right)$,
$x^3=\left(\begin{array}{c}1\\5\end{array}\right)$, and $\omega_1=\omega_2=\omega_3=1$.  
Solving (SDP), we obtain
\[
\gamma_1^*=0.8> 0.5=\frac{1}{n}.
\]
Moreover, the above value of $\gamma_1^*$ is unique.
For the detail of the verification, we refer the reader to \cite{WX16}.
\end{exam}

\section{Approximation algorithm without convex relaxation}

In this section, we propose a simple randomized approximation algorithm for solving (P). We remove the SDP relaxation from Algorithm 1 and then replace the optimal solution $X^*$ with the scalar matrix $n^{-\frac{2}{p}}I_{n}$. It turns out that our new algorithm just uniformly and randomly pick a solution from $\{n^{-\frac{1}{p}}(t_1,t_2,\ldots,t_n)^T\mid~ t_i=-1~{\rm or}~1,~i=1,\ldots,n\}$, which is a set of $2^n$ points on the surface of the unit $\ell_p$-ball. The detailed algorithm is as follows.

~

\begin{center}
\fbox{\shortstack[l]{
{\bf Algorithm 2: Approximation algorithm without convex relaxation.}\\
1.~ Input $\rho\in(0,1)$ and $x^i$  for $i=1,\ldots,m$. Let $\alpha=\sqrt{2\ln(m/\rho)}$. \\
 2.~ Repeatedly generate $\xi\in\Bbb R^n$ with independent $\xi_i$ taking  $\pm1$ with equal \\~~~~~probability until  $(x^i)^T\xi<\alpha\|x^i\|$ for $i\in\{1,\ldots,m\}\setminus\{k\mid\|x^k\|=0\}$.\\
3.~ Output $\widetilde{x}=n^{-\frac{1}{p}}\xi.$
}}
\end{center}

~

Surprisingly, for any $p\ge 2$, we can show that our new Algorithm 2 always provides the approximation bound (\ref{bd}) for (P).
\begin{thm}\label{thm1}
Let $p\ge 2$. For the solution $\widetilde{x}$ returned by Algorithm 2,  we have
\[
 v({\rm P})\geq f(\widetilde{x})>\frac{1-\sqrt{\frac{2}{n}\ln(m/\rho)}}{2}\cdot v({\rm P}).
\]
\end{thm}
\begin{proof}
According to the settings in Algorithm 2, for $i\in\{1,\ldots,m\}$ such that $\|x^i\|>0$,
we have \begin{eqnarray}
\|\widetilde{x}-x^i\|_2^2&=&\|\widetilde{x}\|_2^2-2(x^i)^T\widetilde{x}+\|x^i\|_2^2\nonumber\\
&\geq &\|\tilde{x}\|_2^2-2\alpha\|x^i\|_2\cdot n^{-\frac{1}{p}}+\|x^i\|_2^2\nonumber\\
&=&n^{1-\frac{2}{p}}-2\alpha\|x^i\|_2\cdot n^{-\frac{1}{p}}+\|x^i\|_2^2\nonumber\\
&\ge&(n^{1-\frac{2}{p}}+\|x^i\|_2^2)(1-\alpha\cdot n^{-\frac{1}{2}}),\label{num:1}
\end{eqnarray}
where the inequality (\ref{num:1}) holds since according to the Cauchy-Schwarz inequality, it holds that
\begin{eqnarray}
2\|x^i\|_2\cdot n^{\frac{1}{2}-\frac{1}{p}}\leq n^{1-\frac{2}{p}}+\|x^i\|_2^2.\label{num:0}
\end{eqnarray}
If there is an index $k$ such that $\|x^k\|_2=0$, we have
\begin{equation}
\|\tilde{x}-x^k\|_2^2=n^{1-\frac{2}{p}}>(n^{1-\frac{2}{p}}+\|x^k\|_2^2)(1-\alpha\cdot n^{-\frac{1}{2}}).\label{num:00}
\end{equation}
Next, let $x^*$ be an optimal solution of (P). Then, for $i=1,\ldots,m$ and $\|x^i\|>0$,
it holds that
\begin{eqnarray}
v({\rm P})&\le&{\omega_i}(\|x^*\|_2^2-2(x^i)^Tx^*+\|x^i\|_2^2)\nonumber\\
&\leq&{\omega_i}(n^{1-\frac{2}{p}}-2(x^i)^Tx^*+\|x^i\|_2^2)\label{num:11}\\
&\le&{\omega_i}(n^{1-\frac{2}{p}}+2\|x^i\|_2\cdot\|x^*\|_2+\|x^i\|_2^2)\label{num:4}\\
&\le&{\omega_i}(n^{1-\frac{2}{p}}+2\|x^i\|_2n^{\frac{1}{2}-\frac{1}{p}}+\|x^i\|_2^2)\nonumber\\
&\leq&2{\omega_i}(n^{1-\frac{2}{p}}+\|x^i\|_2^2)\label{num:5},
\end{eqnarray}
where the inequality (\ref{num:11}) holds since it follows from the H\"{o}lder inequality that
\[
\|x^*\|_2^2
\leq \left(\sum_{i=1}^n(x_i^{*2})^{\frac{p}{2}}\right)^{\frac{2}{p}}\left(\sum_{i=1}^n 1^{\frac{p}{p-2}}\right)^{1-\frac{2}{p}}=\left(\sum_{i=1}^n|x^*_i|^{p}\right)^{\frac{2}{p}}
n^{1-\frac{2}{p}}\le n^{1-\frac{2}{p}},
\]
the inequality (\ref{num:4}) holds due to the Cauchy-Schwarz inequality and the inequality (\ref{num:5})
follows from (\ref{num:0}). For the case that there is an index $k$ such that $\|x^k\|=0$, we also have
\begin{equation}
v({\rm P})\leq\omega_k(n^{1-\frac{2}{p}}-2(x^k)^Tx^*+\|x^k\|_2^2)=\omega_k\cdot n^{1-\frac{2}{p}}<2\omega_k(n^{1-\frac{2}{p}}+\|x^k\|_2^2).\label{num:6}
\end{equation}

Thus, it follows from  (\ref{num:1}), (\ref{num:00}),  (\ref{num:5}) and (\ref{num:6}) that
\begin{equation}
\min_{i=1,\ldots,m} \omega_i\|\tilde{x}-x^i\|_2^2>\frac{1-\alpha\cdot n^{-\frac{1}{2}}}{2}\cdot v({\rm P}). \label{final}
\end{equation}
Substituting $\alpha=\sqrt{2\ln(m/\rho)}$ in (\ref{final}) completes the proof.
\end{proof}

\begin{remark}
Theorem \ref{thm1} implies that Algorithm 2 provides a $1/2$ asymptotic approximation bound for (${\rm P}$) as $\frac{n}{\ln(m)}$ increases to infinity.
\end{remark}

\section{Numerical Experiments}
In this section, we numerically compare our new
simple approximation algorithm (Algorithm 2) with the SDP-based algorithm proposed in \cite{HA13} (i.e., Algorithm 1 in this paper) for solving (P).
In Algorithm 1,  we use SDPT3 within CVX \cite{GrB} to solve  $({\rm SDP})$.
All the numerical tests are constructed in MATLAB R2013b and carried out on a laptop computer with $2.3$ GHz processor and $4$ GB RAM.

First, we fix $p=3$, $n=20$, and $\omega_1=\omega_2=\ldots=\omega_m=1$.
Then, we randomly generate the test instances where $m$ varies in $\{10, 12, \ldots, 30\}$.
Since all of the input points $x^i~(i=1,\ldots,m)$ with $m=10,12,\ldots,30$ form an orderly $n\times 220$ matrix. We generate this random matrix  using the following Matlab scripts:
\begin{verbatim}
              rand('state',0); X = 2*rand(n,220)-1;
\end{verbatim}
For each  instance, we independently run each of the two algorithms
$20$ times with the same setting $\rho=0.9999$ and then plot the objective values of the returned approximation solutions in Figure 1.

Second, we fix $n=20,~m=30$,  $\omega_1=\omega_2=\ldots=\omega_m=1$, and let $p=2,3,5,10,20,50,100,200,500,800,1000$. The total $30$ input points $x^i~(i=1,\ldots,m)$ are generated using the following Matlab scripts:
\begin{verbatim}
              rand('state',0); X = 2*rand(n,30)-1;
\end{verbatim}
Both Algorithm 1 and Algorithm 2 are then independently implemented $20$ times for each $p$ with the same setting $\rho=0.9999$. We plot the objective values of the returned approximation solutions in Figure 2.

According to Figures 1 and 2,  the qualities of the approximation solutions returned by Algorithm 2 are in general much higher than those generated by Algorithm 1. The practical performance demonstrates that at least for finding approximation solutions of (P), the SDP relaxation is misleading.

\begin{figure}[h]
\centering
  \includegraphics[width=9cm]{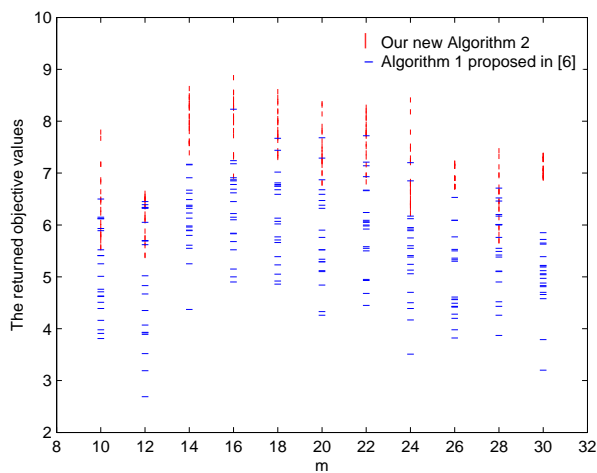}
  \caption{Approximation solutions returned by the SDP-based Algorithm 1 and our new simple Algorithm 2 with different $m$.
 }
  \label{fig1}
\end{figure}
 \begin{figure}[h]
\centering
  \includegraphics[width=9cm]{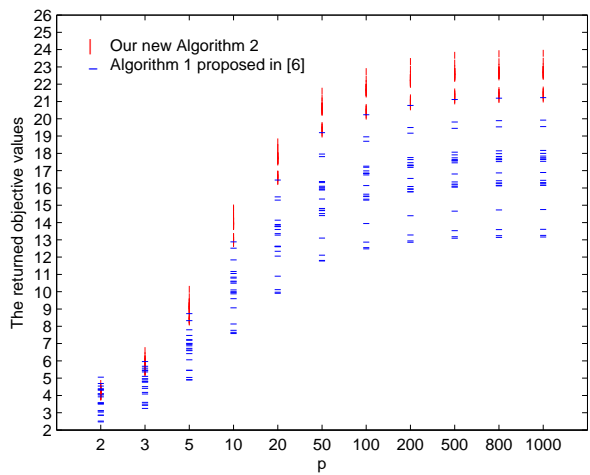}
  \caption{Approximation solutions returned by the SDP-based Algorithm 1 and our new simple Algorithm 2 with different $p$.
 }
  \label{fig2}
\end{figure}

\section{Conclusions}
In this paper, we propose a new simple approximation algorithm for the $\ell_p$-ball constrained weighted maximin dispersion problem (P). It is inherited from the existing SDP-based algorithm by removing the SDP relaxation and then trivially replacing the optimal solution of the SDP relaxation with a particular scalar matrix. Surprisingly,
the simplified algorithm can provide the first unified approximation bound of $\frac{1-O\left(\sqrt{ \ln(m)/n}\right)}{2}$ for any $2\le p\le +\infty$, which remains open up to now except for the special cases $p=2$ and $p=+\infty$. Numerical results also imply that the SDP relaxation technique is misleading in approximately solving (P). Finally, we raise a question whether the unified approximation bound can be extended to  (P) with $1\le p<2$.


\begin{thebibliography}{99999}

%

\bibitem{Ben2002}
 Ben-Tal, A., Nemirovski A., Roos, C.: Robust solutions of
  uncertain quadratic and conic-quadratic problems. SIAM J. Optim., {\bf{13}}, 535--560 (2002)

\bibitem{DW}
  Dasarathy, B., White, L.J.:  A  maximin location problem. Oper. Res. \textbf{28}, 1385--1401 (1980)

\bibitem{GJ}
G\"{a}rtner, B., Matousek, J.: Approximation Algorithms and Semidefinite Programming. Diss. Springer Berlin Heidelberg, 2012.

\bibitem{GW95}  Goemans, M.X.,  Williamson, D.P.: Improved approximation algorithms for maximum cut and satisfiability problems using semidefinite programming. J. ACM, \textbf{42}, 1115--1145  (1995)

\bibitem{GrB}
Grant, M., Boyd, S.: CVX: Matlab software for disciplined convex programming, version 2.1, http://cvxr.com/cvx, June 2015.

\bibitem{HA13}
Haines, S., Loeppky, J., Tseng, P., Wang, X.:
 Convex relaxations of the weighted maxmin dispersion problem.
SIAM J. Optim.
\textbf{23}(4),  2264--2294 (2013)

 \bibitem{JMY}
Johnson, M.E., Moore, L.M., Ylvisaker, D.:  Minimax and maximin distance designs.
J. Satist. Plan. Inference. \textbf{26}, 131--148 (1990)

\bibitem{Luo}
 Luo, Z.Q.,  Ma, W.K.,  So, A.M.C.,  Ye, Y.,  Zhang, S.Z.: Semidefinite relaxation of
quadratic optimization problems. IEEE Signal Process. Mag., \textbf{27}(3),  20--34  (2010)

\bibitem{Nem}
 Nemirovski, A.,  Roos, C.,  Terlaky, T.: On maximization of quadratic form over
intersection of ellipsoids with common center, Math. Program. \textbf{86},
463--473 (1999)

\bibitem{Ne}
 Nesterov, Y.: Semidefinite relaxation and nonconvex quadratic optimization, Optim. Methods Softw., \textbf{9},  141--160 (1998)

\bibitem{NN}
Nesterov, Y.,  Nemirovsky, A.: Interior-Point Polynomial Methods in Convex Programming, SIAM, Philadelphia, PA, 1994.

\bibitem{Sc}
Schaback, R.:  Multivariate Interpolation and Approximation by Translates of a Basis Function, in Approximation Theory VIII,  C. K. Chui and L. L. Schumaker eds., World Scientific, Singapore, 491--514
(1995)

\bibitem{W}
 White, D.J.:  A heuristic approach to a weighted maxmin dispersion problem. IMA J. Math. Appl. Business
and Industry. \textbf{7}, 219--231 (1996)


\bibitem{WX16}
Wang, S., Xia, Y.: On the Ball-Constrained Weighted Maximin Dispersion Problem. SIAM J. Optim., in press (2016)

\bibitem{Ye}
 Ye, Y.: Approximating quadratic programming with bound and quadratic constraints, Math. Program., \textbf{84},  219--226 (1999)
\end{thebibliography}
\end{document}